\documentclass[a4paper,11pt]{amsart}
\usepackage{amsmath,amssymb,amsthm}
\usepackage[all]{xy}

\theoremstyle{plain}
\newtheorem{thm}{Theorem}[section]

\newtheorem{prop}[thm]{Proposition}
\newtheorem{cor}[thm]{Corollary}

\theoremstyle{definition}
\newtheorem{defn}[thm]{Definition}
\newtheorem{rem}[thm]{Remark}
\newtheorem{ex}[thm]{Example}

\numberwithin{equation}{section}

\newcommand{\Q}{\mathbb{Q}}
\newcommand{\bk}{\mathbf{k}}
\newcommand{\bl}{\mathbf{l}}
\newcommand{\frH}{\mathfrak{H}}
\newcommand{\wt}[1]{\lvert{#1}\rvert}
\newcommand{\bast}{\mathbin{\bar{*}}}
\newcommand{\rev}[1]{\overleftarrow{#1}}
\DeclareMathOperator{\re}{Re}

\title{A note on Kawashima functions}
\author{Shuji Yamamoto}
\address{Keio Institute of Pure and Applied Sciences (KiPAS), 
Graduate School of Science and Technology, Keio University, 
3-14-1 Hiyoshi, Kohoku-ku, Yokohama, 223-8522, Japan}
\email{yamashu@math.keio.ac.jp}


\begin{document}
\begin{abstract}
This note is a survey of results on the function $F_\bk(z)$ 
introduced by G.~Kawashima, and its applications 
to the study of multiple zeta values. 
We stress the viewpoint that the Kawashima function is a generalization 
of the digamma function $\psi(z)$, and explain how various formulas 
for $\psi(z)$ are generalized. 
We also discuss briefly the relationship of the results on 
the Kawashima functions with the recent work on Kawashima's MZV relation 
by M.~Kaneko and the author. 
\end{abstract}

\maketitle 

\section{Introduction}
In \cite{K1}, G.~Kawashima introduced a family of special functions 
$F_\bk(z)$, where $\bk=(k_1,\ldots,k_r)$ is a sequence of positive integers, 
and proved some remarkable properties of them. 
As an application, he obtained a large class of algebraic relations 
among the multiple zeta values (MZVs), called \emph{Kawashima's relation}. 
Kawashima's relation can be used to derive some of other classes of relations 
(duality, Ohno's relation, quasi-derivation relation and 
cyclic sum formula; see \cite{K1,T,TW}), 
and is expected to imply all algebraic relations. 

In this note, we survey results on these functions $F_\bk(z)$, 
which we call the \emph{Kawashima functions}, 
and their connections with MZVs. 
We stress the viewpoint that \emph{the Kawashima function is 
a multiple version of the digamma function}. 
Recall that the digamma function $\psi(z)$ is defined as 
the logarithmic derivative of the gamma function:  
$\psi(z)=\frac{d}{dz}\log\Gamma(z)$. 
This is one of the well-studied functions in classical analysis. 
Here we list some formulas on $\psi(z)$ 
($\gamma$ denotes the Euler-Mascheroni constant): 
\begin{itemize}
\item Newton series: 
\begin{equation}\label{eq:psi Newton}
\psi(z+1)=-\gamma+\sum_{n=1}^\infty\frac{(-1)^{n-1}}{n}\binom{z}{n}. 
\end{equation}
\item Interpolation property: For an integer $N\geq 0$, 
\begin{equation}\label{eq:psi interpolation}
\psi(N+1)=-\gamma+\sum_{n=1}^N \frac{1}{n}. 
\end{equation}
\item Integral representation: 
\begin{equation}\label{eq:psi integral}
\psi(z+1)=-\gamma+\int_0^1\frac{1-t^z}{1-t}dt. 
\end{equation}
\item Partial fraction series: 
\begin{equation}\label{eq:psi fraction}
\psi(z+1)=-\gamma+\sum_{n=1}^\infty\biggl(\frac{1}{n}-\frac{1}{n+z}\biggr). 
\end{equation}
\item Taylor series: 
\begin{equation}\label{eq:psi Taylor}
\psi(z+1)=-\gamma+\sum_{m=1}^\infty(-1)^{m-1}\zeta(m+1)z^m. 
\end{equation}
\end{itemize}

In \S\ref{subsec:Newton}, we define the Kawashima function $F_\bk(z)$ 
by a Newton series generalizing \eqref{eq:psi Newton}. 
Then we explain how the formulas \eqref{eq:psi interpolation}, 
\eqref{eq:psi integral} and \eqref{eq:psi fraction} are 
extended to $F_\bk(z)$, 
in \S\ref{subsec:interpolation}, \S\ref{subsec:int} and 
\S\ref{subsec:fraction} respectively. 

The Taylor expansion of $F_\bk(z)$ at $z=0$, which generalizes 
\eqref{eq:psi Taylor}, is descirbed in \S\ref{subsec:Taylor}. 
In fact, there are three methods to compute the Taylor coefficients, 
each of which expresses the coefficients in terms of MZVs 
 (Proposition \ref{prop:F_k Taylor1}, Proposition \ref{prop:F_k Taylor2} 
and Corollary \ref{cor:F_k Taylor3}). 
In \S\ref{subsec:harmonic}, we treat another important property 
of Kawashima functions, the \emph{harmonic relation} 
(Theorem \ref{thm:F_k harmonic}). 
Then by combining it with the Taylor series \eqref{eq:F_k Taylor1}, 
we deduce Kawashima's algebraic relation for MZVs 
(Corollary \ref{cor:KawashimaRel}). 

At the Lyon Conference, the author talked on a new proof of 
Kawashima's MZV relation based on the double shuffle relation and 
the regularization theorem, which is a part of the work 
with M.~Kaneko \cite{KY}. 
In \S\ref{subsec:Rem}, we briefly discuss the relationship 
between this proof and the results on Kawashima functions 
presented in \S\ref{subsec:Taylor} and \S\ref{subsec:harmonic}. 

\smallskip 

Though this is basically an expository article on known results 
(largely due to Kawashima), it includes some results 
which appear in print for the first time; 
Proposition \ref{prop:F_k int}, Proposition \ref{prop:F_k inductive} 
and Corollary \ref{cor:F_k diff}. 
On the other hand, we should also note that we leave out some important works 
related with Kawashima functions and Kawashima's MZV relation; 
particularly, their $q$-analogue studied by Takeyama \cite{Take}, 
and the generalization of Kawashima's relation 
to `interpolated' MZVs by Tanaka and Wakabayashi \cite{TW2}. 
For details, we refer the reader to their original articles. 

\section{Definition and Formulas of the Kawashima function}
In this section, we define the Kawashima function by 
generalizing the Newton series in \eqref{eq:psi Newton}, 
and present generalizations of \eqref{eq:psi interpolation}, 
\eqref{eq:psi integral} and \eqref{eq:psi fraction}. 

\subsection{Multiple harmonic sums}
Let $\bk=(k_1,\ldots,k_r)$ be an \emph{index}, i.e., 
a sequence of positive integers of finite length $r$. 
We call $\wt{\bk}:=k_1+\cdots+k_r$ the \emph{weight} of $\bk$. 
We regard the sequence of length $0$ as an index, 
the \emph{empty index} denoted by $\varnothing$, 
though we mainly consider nonempty indices. 

For a nonempty index $\bk=(k_1,\ldots,k_r)$ and an integer $N\geq 0$, we put 
\begin{align*}
s(\bk,N)&=\sum_{0<m_1<\cdots<m_{r-1}<m_r=N}
\frac{1}{m_1^{k_1}\cdots m_r^{k_r}}, \\
s^\star(\bk,N)&=\sum_{0<m_1\leq \cdots\leq m_{r-1}\leq m_r=N}
\frac{1}{m_1^{k_1}\cdots m_r^{k_r}}, \\
S(\bk,N)&=\sum_{0<m_1<\cdots<m_{r-1}<m_r\leq N}
\frac{1}{m_1^{k_1}\cdots m_r^{k_r}}=\sum_{n=1}^Ns(\bk,n), \\
S^\star(\bk,N)&=\sum_{0<m_1\leq \cdots\leq m_{r-1}\leq m_r\leq N}
\frac{1}{m_1^{k_1}\cdots m_r^{k_r}}=\sum_{n=1}^Ns^\star(\bk,n). 
\end{align*}

In \cite{Y}, integral representations of $s^\star(\bk,N)$ 
and $S^\star(\bk,N)$ are given: 
\begin{thm}\label{thm:S_k int}
For a nonempty index $\bk=(k_1,\ldots,k_r)$, put $k=\wt{\bk}$ and 
\begin{align*}
A(\bk)&=\{k_1,k_1+k_2,\ldots,k_1+\cdots+k_{r-1}\},\\
\Delta(\bk)&=\Biggl\{(t_1,\ldots,t_k)\in (0,1)^k\Biggm|
\begin{array}{l}
t_j>t_{j+1} \text{ if $j\notin A(\bk)$},\\ 
t_j<t_{j+1} \text{ if $j\in A(\bk)$}
\end{array}\Biggr\}. 
\end{align*}
Then we have 
\begin{align}
\label{eq:s_k int} 
s^\star(\bk,N)&=\int_{\Delta(\bk)}\omega_{\delta(1)}(t_1)
\cdots\omega_{\delta(k-1)}(t_{k-1})\,t_k^{N-1}dt_k, \\
\label{eq:S_k int} 
S^\star(\bk,N)&=\int_{\Delta(\bk)}\omega_{\delta(1)}(t_1)
\cdots\omega_{\delta(k-1)}(t_{k-1})\frac{1-t_k^N}{1-t_k}dt_k, 
\end{align}
where $\omega_0(t)=\frac{dt}{t}$, $\omega_1(t)=\frac{dt}{1-t}$ and 
\[\delta(j)=\begin{cases}0 & \text{if $j\notin A(\bk)$}, \\
1 & \text{if $j\in A(\bk)$}. \end{cases}\]
\end{thm}
\begin{proof}
The first formula \eqref{eq:s_k int} is \cite[Theorem 1.2]{Y}, 
stated in different symbols (in \cite{Y}, 
the inverse order is adopted for the index). 
The second \eqref{eq:S_k int} is an immediate consequence of the first, 
since $\sum_{n=1}^N t_k^{n-1}=\frac{1-t_k^N}{1-t_k}$. 
\end{proof}

As noted in \cite{Y}, the integral representations \eqref{eq:s_k int} 
and \eqref{eq:S_k int} imply the following identities, known as 
\emph{Hoffman's duality}: 
\begin{thm}[\cite{H,K1}]
Let $\bk^\vee$ be the \emph{Hoffman dual} of $\bk$, 
i.e., the index characterized by 
\[\wt{\bk}=\wt{\bk^\vee},\quad 
A(\bk)\amalg A(\bk^\vee)=\{1,2,\ldots,\wt{\bk}-1\}. \]
Then we have 
\begin{align}
\label{eq:s_k dual} 
s^\star(\bk,N)&=\sum_{n=1}^N(-1)^{n-1}s^\star(\bk^\vee,n)\binom{N-1}{n-1}, \\
\label{eq:S_k dual} 
S^\star(\bk,N)&=\sum_{n=1}^N(-1)^{n-1}s^\star(\bk^\vee,n)\binom{N}{n}. 
\end{align}
\end{thm}
\begin{proof}
Under the change of variables $t_i\mapsto 1-t_i$, 
$\omega_0(t_i)$ and $\omega_1(t_i)$ are interchanged and 
$\Delta(\bk)$ maps onto $\Delta(\bk^\vee)$. 
Hence the identities follow from 
\begin{align*}
(1-t_k)^{N-1}&=\sum_{n=1}^N(-t_k)^{n-1}\binom{N-1}{n-1}, \\
\frac{1-(1-t_k)^N}{1-(1-t_k)}&=\sum_{n=1}^N(-t_k)^{n-1}\binom{N}{n}. \qedhere
\end{align*}
\end{proof}

\subsection{Newton series (definition)}\label{subsec:Newton}
Following Kawashima \cite{K1}, we define the Kawashima function 
by a Newton series: 
\begin{defn}\label{defn:F_k}
For a nonempty index $\bk$, we define the \emph{Kawashima function} 
$F_\bk(z)$ as 
\begin{equation}\label{eq:F_k defn}
F_\bk(z)=\sum_{n=1}^\infty(-1)^{n-1}s^\star(\bk^\vee,n)\binom{z}{n}. 
\end{equation}
As a convention, we put $F_\varnothing(z)=1$. 
\end{defn}

From the Newton series formula for the digamma function 
\eqref{eq:psi Newton}, we see that $F_1(z)=\psi(z+1)+\gamma$. 
Hence the Kawashima function may be viewed as a generalization 
of (a slight modification of) the digamma function. 

With regard to the convergence of the series \eqref{eq:F_k defn}, 
Kawashima proved: 

\begin{prop}[{\cite[Proposition 5.1]{K1}}]
Let $\bk$ be a nonempty index and $\rho$ the last component of 
the Hoffman dual of $\bk$. 
Then the Newton series $F_\bk(z)$ has the abscissa of convergence $-\rho$, 
i.e., converges uniformly on compact sets in the half plane $\re(z)>-\rho$, 
and diverges on $\re(z)<-\rho$. 
\end{prop}

In particular, all Kawashima functions are defined and 
holomorphic on $\re(z)>-1$. Hence, at least, it makes sense to consider 
the Taylor expansion at $z=0$. 
We present explicit results in \S\ref{subsec:Taylor}. 

\begin{rem}
If we write $\bk=(k_1,\ldots,k_q,\underbrace{1,\ldots,1}_{l})$, 
where $k_q>1$ or $q=0$, then $\rho$ is given by 
\[\rho=\begin{cases}
l+1 & \text{if $q\geq 1$},\\
l & \text{if $q=0$}. 
\end{cases}\]
In \cite[Proposition 5.1]{K1}, the latter case seems to be missed. 
\end{rem}

\subsection{Interpolation property}\label{subsec:interpolation}
\begin{prop}\label{prop:F_k interpolation}
For any integer $N\geq 0$, we have 
\begin{equation}\label{eq:F_k interpolation}
F_\bk(N)=S^\star(\bk,N). 
\end{equation}
Conversely, if a Newton series $f(z)=\sum_{n=0}^\infty a_n\binom{z}{n}$ 
satisfies $f(N)=S^\star(\bk,N)$ for all $N\geq 0$, 
then $f(z)$ coincides with $F_\bk(z)$ coefficientwise 
(i.e., $a_n=(-1)^{n-1}s^\star(\bk^\vee,n)$ hold for all $n$). 
\end{prop}
\begin{proof}
The identity \eqref{eq:F_k interpolation} follows from \eqref{eq:S_k dual}. 
For the second assertion, note the fact that the identity 
\[f(N)=\sum_{n=0}^N a_n\binom{N}{n}\]
determines inductively the coefficients $a_n$ by the values $f(N)$. 
\end{proof}

This characterization of the Kawashima function 
by its values at non-negative integers plays an essential role 
in Kawashima's proofs of the fraction series expansion 
(Theorem \ref{thm:F_k fraction}) 
and the harmonic relation (Theorem \ref{thm:F_k harmonic}). 

\subsection{Integral representation}\label{subsec:int}
\begin{prop}\label{prop:F_k int}
With the same notation as in Theorem \ref{thm:S_k int}, 
we have 
\begin{equation}\label{eq:F_k int}
F_\bk(z)=\int_{\Delta(\bk)}\omega_{\delta(1)}(t_1)
\cdots\omega_{\delta(k-1)}(t_{k-1})\frac{1-t_k^z}{1-t_k}dt_k. 
\end{equation}
\end{prop}
\begin{proof}
Just as in the proof of \eqref{eq:S_k dual}, 
make the change of variables $t_i\mapsto 1-t_i$ 
and use the identity 
\[\frac{1-(1-t_k)^z}{1-(1-t_k)}
=\sum_{n=1}^\infty (-t_k)^{n-1}\binom{z}{n}. \qedhere\]
\end{proof}

\begin{ex}
Let us describe the relation between the polygamma function 
$\psi^{(m)}(z)=\bigl(\frac{d}{dz}\bigr)^m\psi(z)$ and 
the Kawashima function. For $m=0$, we already know that 
$F_1(z)=\psi^{(0)}(z+1)+\gamma$. 
For $m>0$, we have 
\[\psi^{(m)}(z+1)
=\biggl(\frac{d}{dz}\biggr)^m\int_0^1\frac{1-t^z}{1-t}dt
=-\int_0^1(\log t)^m\frac{t^z}{1-t}dt.\]
Since 
\[(\log t)^m=\biggl(-\int_t^1\frac{du}{u}\biggr)^m
=(-1)^mm!\int_{1>u_1>\cdots>u_m>t}\frac{du_1}{u_1}\cdots\frac{du_m}{u_m}, \]
we have 
\begin{align*}
\psi^{(m)}(z+1)&=(-1)^{m-1}m!\int_{1>u_1>\cdots>u_m>t>0}
\frac{du_1}{u_1}\cdots\frac{du_m}{u_m}\frac{t^z}{1-t}dt\\
&=(-1)^mm!\bigl(F_{m+1}(z)-\zeta(m+1)\bigr). 
\end{align*}
Here we use the integral representation 
\eqref{eq:F_k int} for $F_{m+1}(z)$ together with 
the iterated integral expression 
\begin{equation}\label{eq:zeta int}
\zeta(m+1)=\int_{1>u_1>\cdots>u_m>t>0}
\frac{du_1}{u_1}\cdots\frac{du_m}{u_m}\frac{1}{1-t}dt. 
\end{equation}
Hence we get 
\begin{equation}\label{eq:polygamma}
F_{m+1}(z)=\frac{(-1)^m}{m!}\psi^{(m)}(z+1)+\zeta(m+1)
\end{equation}
for integers $m>0$. Note that this also holds for $m=0$ if 
we interpret $\zeta(1)$ as $\gamma$. 
\end{ex}

\subsection{Fraction series}\label{subsec:fraction}
Here we give two generalizations of \eqref{eq:psi fraction}. 
The first is an inductive formula: 

\begin{prop}\label{prop:F_k inductive}
Let $\bk=(k_1,\ldots,k_r)$ be a nonempty index and 
write $\bk_-=(k_1,\ldots,k_{r-1})$ 
(when $r=1$, $\bk_-$ is the empty index $\varnothing$). 
Then we have 
\begin{equation}\label{eq:F_k inductive}
F_\bk(z)=\sum_{n=1}^\infty
\biggl(s^\star(\bk,n)-\frac{F_{\bk_-}(n+z)}{(n+z)^{k_r}}\biggr). 
\end{equation}
\end{prop}
\begin{proof}
Put $k=\wt{\bk}$ and $k'=\wt{\bk_-}$. 
Then the tail of the multiple integral \eqref{eq:F_k int} is written as 
\begin{align*}
&\frac{dt_{k'}}{1-t_{k'}}
\int_{t_{k'}}^1\frac{dt_{k'+1}}{t_{k'+1}}
\int_0^{t_{k'+1}}\frac{dt_{k'+2}}{t_{k'+2}}\\
&\hspace{100pt}\cdots
\int_0^{t_{k-2}}\frac{dt_{k-1}}{t_{k-1}}
\int_0^{t_{k-1}}\frac{1-t_k^z}{1-t_k}dt_k\\
&=\sum_{n=1}^\infty
\frac{dt_{k'}}{1-t_{k'}}
\int_{t_{k'}}^1\frac{dt_{k'+1}}{t_{k'+1}}
\int_0^{t_{k'+1}}\frac{dt_{k'+2}}{t_{k'+2}}\\
&\hspace{100pt}\cdots
\int_0^{t_{k-2}}\frac{dt_{k-1}}{t_{k-1}}
\int_0^{t_{k-1}}(t_k^{n-1}-t_k^{n+z-1})dt_k\\
&=\sum_{n=1}^\infty\frac{dt_{k'}}{1-t_{k'}}
\biggl(\frac{1-t_{k'}^n}{n^{k_r}}-\frac{1-t_{k'}^{n+z}}{(n+z)^{k_r}}\biggr). 
\end{align*}
Hence the whole integral is equal to 
\begin{align*}
&\sum_{n=1}^\infty \int_{\Delta(\bk_-)}
\prod_{j=1}^{k'-1}\omega_{\delta(j)}(t_j)
\biggl(\frac{1}{n^{k_r}}\frac{1-t_{k'}^n}{1-t_{k'}}
-\frac{1}{(n+z)^{k_r}}\frac{1-t_{k'}^{n+z}}{1-t_{k'}}\biggr)dt_{k'}\\
&=\sum_{n=1}^\infty\biggl(\frac{F_{\bk_-}(n)}{n^{k_r}}
-\frac{F_{\bk_-}(n+z)}{(n+z)^{k_r}}\biggr)
=\sum_{n=1}^\infty
\biggl(s^\star(\bk,n)-\frac{F_{\bk_-}(n+z)}{(n+z)^{k_r}}\biggr). \qedhere
\end{align*}
\end{proof}

For $\bk=(1)$, the above formula \eqref{eq:F_k inductive} is the same as 
the formula \eqref{eq:psi fraction} for the digamma function. 
See Example \ref{ex:fraction r=1} below. 

\begin{cor}\label{cor:F_k diff}
With the same notation as in Proposition \ref{prop:F_k inductive}, 
Kawashima functions satisfy the difference equation 
\begin{equation}\label{eq:F_k diff}
F_\bk(z)-F_\bk(z-1)=\frac{F_{\bk_-}(z)}{z^{k_r}}. 
\end{equation}
\end{cor}
\begin{proof}
Since both sides are analytic, we may assume that $z$ is real. 
From Proposition \ref{prop:F_k inductive}, we obtain 
\[F_\bk(z)-F_\bk(z-1)=\frac{F_{\bk_-}(z)}{z^{k_r}}
-\lim_{n\to\infty}\frac{F_{\bk_-}(n+z)}{(n+z)^{k_r}}, \]
hence the proposition follows from that 
\[\frac{F_{\bk_-}(z)}{z^{k_r}}\to 0 \quad (z\to\infty). \]
Moreover, from Proposition \ref{prop:F_k int}, 
we see that $F_\bk(z)$ is monotone increasing for $z\geq 0$. 
Therefore, it suffices to show that 
\[\frac{F_{\bk_-}(N)}{N^{k_r}}=s^\star(\bk,N)\to 0\quad (N\to\infty). \]
Now we have an estimate 
\[0\leq s^\star(\bk,N)\leq s^\star\bigl((\underbrace{1,\ldots,1}_r),N\bigr)
=\frac{1}{N}\sum_{n=1}^N 
s^\star\bigl((\underbrace{1,\ldots,1}_{r-1}),n\bigr), \]
and the statement is proven by induction on $r$. 
\end{proof}

Note that, from \eqref{eq:F_k diff}, it follows that 
$F_\bk(z)$ is meromorphically continued to the whole complex plane. 

\bigskip 

The second generalization of \eqref{eq:psi fraction}, 
which seems more nontrivial than the first, 
was given by Kawashima \cite{K2}. 
To present it, we make some definitions. 

For a nonempty index $\bk=(k_1,\ldots,k_r)$, 
write $\rev{\bk}=(k_r,\ldots,k_1)$. 

\begin{defn}
For integers $r>0$ and $n_1,\ldots,n_r>0$, we put 
\begin{align*}
P_r(n_1,\ldots,n_r;z)&=\frac{1}{n_1\cdots n_{r-1}(n_r+z)}, \\
\tilde{P}_r(n_1,\ldots,n_r;z)
&=\frac{1}{n_1\cdots n_{r-1}}\biggl(\frac{1}{n_r}-\frac{1}{n_r+z}\biggr). 
\end{align*}
Then, for a nonempty index $\bk=(k_1,\ldots,k_r)$ of weight $k=\wt{\bk}$, 
we define 
\begin{equation}\label{eq:G_k}
\begin{split}
G_\bk(z)=
\sum &P_{k_1}(n_1,\ldots,n_{k_1};z)P_{k_2}(n_{k_1+1},\ldots,n_{k_1+k_2};z)
\cdots \\
&\cdot P_{k_{r-1}}(n_{k_1+\cdots+k_{r-2}+1},\ldots,n_{k_1+\cdots+k_{r-1}};z)\\
&\cdot \tilde{P}_{k_r}(n_{k_1+\cdots+k_{r-1}+1},\ldots,n_k;z), 
\end{split}
\end{equation}
where the sum is taken over all sequences of positive integers 
$n_1,\ldots,n_k$ satisfying 
\begin{equation}\label{eq:summation}
\begin{cases}
n_j<n_{j+1} & \text{if $j\notin A(\bk)$}, \\
n_j\leq n_{j+1} & \text{if $j\in A(\bk)$} 
\end{cases}
\end{equation}
(recall that $A(\bk)$ denotes the set 
$\{k_1,k_1+k_2,\ldots,k_1+\cdots+k_{r-1}\}$).  
\end{defn}

For example, 
\[G_{1,3}(z)=\sum_{0<n_1\leq n_2<n_3<n_4}
\frac{1}{(n_1+z)n_2n_3}\biggl(\frac{1}{n_4}-\frac{1}{n_4+z}\biggr). \]
By the following theorem, this is equal to $F_{1,1,2}(z)$. 

\begin{thm}[{\cite[Theorem 4.4]{K2}}]\label{thm:F_k fraction}
For a nonempty index $\bk$, we have 
\begin{equation}\label{eq:F_k fraction}
F_\bk(z)=G_{\rev{\bk^\vee}}(z). 
\end{equation}
\end{thm}

\begin{ex}\label{ex:fraction r=1}
Let us consider an index $\bk=(k)$ of length $1$. 
By \eqref{eq:F_k inductive} and \eqref{eq:F_k fraction}, we have 
\begin{align*}
F_k(z)&=\sum_{n=1}^\infty\biggl(\frac{1}{n^k}-\frac{1}{(n+z)^k}\biggr)\\
=G_{\underbrace{\scriptstyle 1,\ldots,1}_k}(z)
&=\sum_{0<n_1\leq\cdots\leq n_k}
\frac{1}{(n_1+z)\cdots (n_{k-1}+z)}
\biggl(\frac{1}{n_k}-\frac{1}{n_k+z}\biggr). 
\end{align*}
In particular, when $k=1$, 
both expressions coincide with the formula 
\eqref{eq:psi fraction} for the digamma function. 
For $k>1$, in contrast, it seems not easy to see that 
these two expressions are equal. 
\end{ex}

\section{Kawashima's relation of multiple zeta values}
In this section, we discuss the connections of 
Kawashima functions with multiple zeta values. 

\subsection{Notation related to multiple zeta values}
A nonempty index $\bk=(k_1,\ldots,k_r)$ is said \emph{admissible} 
if $k_r>1$. For such $\bk$, we define the multiple zeta value (MZV) 
and the multiple zeta-star value (MZSV) by 
\begin{align}
\zeta(\bk)
&=\sum_{0<m_1<\cdots<m_r}\frac{1}{m_1^{k_1}\cdots m_r^{k_r}}
=\sum_{n=1}^\infty s(\bk,n), \\
\zeta^\star(\bk)
&=\sum_{0<m_1\leq\cdots\leq m_r}\frac{1}{m_1^{k_1}\cdots m_r^{k_r}}
=\sum_{n=1}^\infty s^\star(\bk,n). 
\end{align}
We also regard the empty index $\varnothing$ as admissible, 
and put $\zeta(\varnothing)=\zeta^\star(\varnothing)=1$. 

Let $\frH^1=\bigoplus_{\bk}\Q\cdot\bk$ be the $\Q$-vector space 
freely generated by all indices $\bk$, 
and $\frH^0$ the subspace generated by the admissible indices. 
There are two $\Q$-bilinear products $*$ and $\bast$, 
called the \emph{harmonic products}, 
for which $\varnothing$ is the unit element and which satisfies 
\begin{align*}
\bk*\bl&=(\bk_-*\bl,k_r)+(\bk*\bl_-,l_s)+(\bk_-*\bl_-,k_r+l_s),\\
\bk\bast\bl&=(\bk_-\bast\bl,k_r)+(\bk\bast\bl_-,l_s)
-(\bk_-\bast\bl_-,k_r+l_s),
\end{align*}
where $\bk=(k_1,\ldots,k_r)$ and $\bl=(l_1,\ldots,l_s)$ are 
any nonempty indices and 
$\bk_-=(k_1,\ldots,k_{r-1})$, $\bl_-=(l_1,\ldots,l_{s-1})$. 
In the following, we also need another product 
\[\bk\circledast\bl=\bigl(\bk_-*\bl_-,k_r+l_s\bigr), \]
defined on the subspace of $\frH^1$ generated by all nonempty indices. 

We extend the map $\bk\mapsto s(\bk,z)$ to a linear map on $\frH^1$. That is, 
for $v=\sum_{\bk}a_\bk\cdot\bk\in\frH^1$, we put 
\[s(v,z)=\sum_\bk a_\bk\, s(\bk,z). \]
The same rule also applies to $S(\bk,N)$, $F_\bk(z)$, $\zeta(\bk)$ and so on. 
Then one can see that 
\begin{align*}
s(v,N)s(w,N)&=s(v\circledast w,N), \\
S(v,N)S(w,N)&=S(v*w,N), & 
S^\star(v,N)S^\star(w,N)&=S^\star(v\bast w,N), \\ 
\zeta(v)\zeta(w)&=\zeta(v*w), & 
\zeta^\star(v)\zeta^\star(w)&=\zeta^\star(v\bast w). 
\end{align*}
Moreover, we define a linear operator $v\mapsto v^\star$ on $\frH^1$ by 
\[\begin{split}
&(k_1,\ldots,k_r)^\star\\
&=\sum_{0<j_1<\cdots<j_q=r}
\bigl(k_1+\cdots+k_{j_1},k_{j_1+1}+\cdots+k_{j_2},
\ldots,k_{j_{q-1}+1}+\cdots+k_{j_q}\bigr),\end{split}\]
so that $s^\star(v,N)=s(v^\star,N)$, 
$S^\star(v,N)=S(v^\star,N)$ and $\zeta^\star(v)=\zeta(v^\star)$. 

\subsection{Taylor series}\label{subsec:Taylor}
We give three ways to express the Taylor coefficients of $F_\bk(z)$ 
at $z=0$ in terms of MZVs. The first is to substitute 
\begin{align*}
\binom{z}{n}&=\frac{z(z-1)(z-2)\cdots(z-n+1)}{n!}\\
&=\biggl(\frac{z}{1}-1\biggr)\biggl(\frac{z}{2}-1\biggr)
\cdots\biggl(\frac{z}{n-1}-1\biggr)\frac{z}{n}\\
&=\sum_{m=1}^n\sum_{0<a_1<\cdots<a_m=n}\frac{(-1)^{n-m}z^m}{a_1\cdots a_m}\\
&=\sum_{m=1}^n(-1)^{n-m}s\bigl((\underbrace{1,\ldots,1}_m),n\bigr)z^m
\end{align*}
into the definition \eqref{eq:F_k defn} of $F_\bk(z)$. 
The result is: 

\begin{prop}[{\cite[Proposition 5.2]{K1}}]\label{prop:F_k Taylor1}
For any nonempty index $\bk$, the Taylor expansion of 
$F_\bk(z)$ at $z=0$ is given by 
\begin{equation}\label{eq:F_k Taylor1}
F_\bk(z)=\sum_{m=1}^\infty (-1)^{m-1}
\zeta\bigl((\underbrace{1,\ldots,1}_{m})\circledast(\bk^\vee)^\star\bigr)
z^m. 
\end{equation}
\end{prop}

The second method is to differentiate repeatedly the integral 
representation \eqref{eq:F_k int} as in the proof of \eqref{eq:polygamma}. 
By this method, we obtain the following formula. 

\begin{prop}\label{prop:F_k Taylor2}
With the same notation as in Theorem \ref{thm:S_k int}, 
we put 
\[A_m(\bk)=\int_{\Delta(\bk,\underbrace{\scriptstyle 1,\ldots,1}_m)}
\omega_{\delta(1)}(t_1)\cdots\omega_{\delta(k-1)}(t_{k-1})
\frac{dt_k}{1-t_k}\frac{dt_{k+1}}{t_{k+1}}\cdots\frac{dt_{k+m}}{t_{k+m}}. \]
Then we have 
\begin{equation}\label{eq:F_k Taylor2}
F_\bk(z)=\sum_{m=1}^\infty(-1)^{m-1}A_m(\bk)z^m. 
\end{equation}
\end{prop}

The third method is based on Theorem \ref{thm:F_k fraction} 
and a computation of the derivatives of $G_\bk(z)$ at $z=0$. 

\begin{defn}
Let $\bk=(k_1,\ldots,k_r)$ be a nonempty index of weight $k=\wt{\bk}$. 
For an index $\bl=(l_1,\ldots,l_k)$ of length $k$, define 
\begin{equation}\label{eq:zeta_k}
\zeta_\bk(\bl)=\sum\frac{1}{n_1^{l_1}\cdots n_k^{l_k}}
\end{equation}
where the sum is taken just as in the definition of $G_\bk(z)$, i.e., 
over all sequences of positive integers $n_1,\ldots,n_k$ satisfying 
\eqref{eq:summation}. 
\end{defn}

\begin{prop}[{\cite[Proposition 5.2]{K2}}]\label{prop:C_m(k)}
For a nonempty index $\bk=(k_1,\ldots,k_r)$ and an integer $m\geq 1$, put 
\[C_m(\bk)=
\sum_{\substack{l_1,\ldots,l_{r-1}\geq 0,l_r\geq 1\\ l_1+\cdots+l_r=m}}
\zeta_\bk(\underbrace{1,\ldots,1}_{k_1-1},l_1+1,\ldots,
\underbrace{1,\ldots,1}_{k_r-1},l_r+1). \]
Then we have 
\begin{equation}\label{eq:C_m(k)}
\frac{G_\bk^{(m)}(0)}{m!}=(-1)^{m-1}C_m(\bk). 
\end{equation}
\end{prop}

\begin{cor}\label{cor:F_k Taylor3}
For a nonempty index $\bk$, we have 
\begin{equation}\label{eq:F_k Taylor3}
F_\bk(z)=\sum_{m=1}^\infty (-1)^{m-1}C_m(\rev{\bk^\vee})z^m. 
\end{equation}
\end{cor}

By comparing the above three expressions of the Taylor expansion 
of $F_\bk(z)$, we get 
\begin{equation}\label{eq:coeff}
\zeta\bigl((\underbrace{1,\ldots,1}_{m})\circledast(\bk^\vee)^\star\bigr)
=A_m(\bk)=C_m(\rev{\bk^\vee}). 
\end{equation}
Since each of these expressions can be written 
as a sum of finitely many MZVs, 
this identity gives linear relations among MZVs. 
The relation 
\begin{equation}\label{eq:coeffA}
\zeta\bigl((\underbrace{1,\ldots,1}_{m})\circledast(\bk^\vee)^\star\bigr)
=A_m(\bk)
\end{equation}
appears in \cite{KY} with a different proof (see \S\ref{subsec:Rem} below), 
while 
\begin{equation}\label{eq:coeffC}
\zeta\bigl((\underbrace{1,\ldots,1}_{m})\circledast\bk^\star\bigr)
=C_m(\rev{\bk})
\end{equation}
($\bk^\vee$ is replaced by $\bk$) is given in \cite[Proposition 5.3]{K2}. 
Kawashima also proved the equivalence of 
\eqref{eq:coeffC} for $m=1$ and the duality relation. 

\begin{ex}
Let us consider the case of $\bk=(1)$. 
Then the formula \eqref{eq:F_k Taylor1} says that 
\[F_1(z)=\sum_{m=1}^\infty
(-1)^{m-1}\zeta(\underbrace{1,\ldots,1}_{m-1},2)z^m. \]
On the other hand, \eqref{eq:F_k Taylor2} and \eqref{eq:F_k Taylor3} give 
\[F_1(z)=\sum_{m=1}^\infty(-1)^{m-1}\zeta(m+1)z^m, \]
which is exactly the classical formula \eqref{eq:psi Taylor} 
in the introduction. 
Hence we obtain 
$\zeta(\underbrace{1,\ldots,1}_{m-1},2)=\zeta(m+1)$, 
which is a special case of the duality. 
\end{ex}

\subsection{Harmonic relation}\label{subsec:harmonic}
\begin{thm}[{\cite[Theorem 5.3]{K1}}]\label{thm:F_k harmonic}
For any indices $\bk$ and $\bl$, we have 
\begin{equation}\label{eq:F_k harmonic}
F_\bk(z)F_\bl(z)=F_{\bk\bast\bl}(z). 
\end{equation}
\end{thm}

By substituting the Taylor expansion \eqref{eq:F_k Taylor1} 
into this relation \eqref{eq:F_k harmonic}, we obtain 
algebraic relations among MZVs.  

\begin{cor}[Kawashima's relation]\label{cor:KawashimaRel}
For any indices $\bk$, $\bl$ and any integer $m\geq 1$, we have 
\begin{equation}\label{eq:KawashimaRel}
\begin{split}
\sum_{\substack{p,q\geq 1\\ p+q=m}}
\zeta\bigl((\underbrace{1,\ldots,1}_p)\circledast(\bk^\vee)^\star\bigr)
&\zeta\bigl((\underbrace{1,\ldots,1}_q)\circledast(\bl^\vee)^\star\bigr)\\
&=-\zeta\bigl((\underbrace{1,\ldots,1}_m)\circledast
((\bk\bast\bl)^\vee)^\star\bigr). 
\end{split}
\end{equation}
\end{cor}

\subsection{Remark on the work of Kaneko-Yamamoto}\label{subsec:Rem}
Here we use the notation of \cite{KY}. 
Then the integral $A_m(\bk)$ in Proposition \ref{prop:F_k Taylor2} 
is written as 
\[A_m(\bk)=I\left(\begin{xy}
{(0,5) \ar @{{o}.} (4,1)}, 
{(4,1) \ar @{{o}-} (8,-3)}, 
{(8,-3) \ar @{{*}-} (14,-3)*+[F]{\bk}}, 
{(0,4) \ar @/_1mm/ @{-}_{m} (3,1)} 
\end{xy}\ \right). \]
This means that $A_m(\bk)$ is identical to the value 
denoted by the same symbol in \cite[\S6]{KY} in the case of $Z=\zeta$. 

By the change of variables $t_i\mapsto 1-t_i$ 
(in other words, by the duality relation for MZVs), 
this integral transforms to 
\begin{equation}\label{eq:A_m duality}
A_m(\bk)=I\left(\begin{xy}
{(0,-5) \ar @{{*}.} (4,-1)}, 
{(4,-1) \ar @{{*}-} (8,3)}, 
{(8,3) \ar @{o-} (15,3)*+[F]{\bk^\vee}}, 
{(0,-4) \ar @/^1mm/ @{-}^{m} (3,-1)} 
\end{xy}\ \right)
=\zeta\bigl(\mu\bigl((\underbrace{1,\ldots,1}_m),\bk^\vee\bigr)\bigr). 
\end{equation}
Hence, by combining with \eqref{eq:coeffA}, we obtain 
\begin{equation}\label{eq:IntSer}
\zeta\bigl((\underbrace{1,\ldots,1}_m)\circledast(\bk^\vee)^\star\bigr)
=\zeta\bigl(\mu\bigl((\underbrace{1,\ldots,1}_m),\bk^\vee\bigr)\bigr). 
\end{equation}
This is a special case of the ``integral-series identity'' 
\cite[Theorem 4.1]{KY}. 
In fact, this special case is sufficient to imply 
the general integral-series identity, under the assumption of 
the double shuffle relation (see \cite[Proposition 5.4]{KY}). 

In the proof of \cite[Theorem 6.7]{KY}, 
Kaneko and the author went in reverse, i.e., 
deduced the equality \eqref{eq:coeffA} 
from the integral-series identity \eqref{eq:IntSer} 
and the duality \eqref{eq:A_m duality}. 
Then they proved the relation 
\[\sum_{\substack{p,q\geq 1\\ p+q=m}}
A_p(\bk)A_q(\bl)=-A_m(\bk\bast\bl)\]
corresponding to Kawashima's relation \eqref{eq:KawashimaRel} 
without using Kawashima functions. 
Indeed, since their aim was to show the statement 
``the regularized double shuffle relation and the duality 
imply Kawashima's relation'' in an algebraic setting, 
transcendental object such as $F_\bk(z)$ was not available. 

\section*{Acknowledgments}
The author would like to thank Prof.~Masanobu Kaneko for valuable discussions. 
He also wishes to express his gratitude to the organizers of 
this Lyon Conference 2016 for their invitation and hospitality. 
This work was supported in part by JSPS KAKENHI 
JP26247004, JP16H06336 and JP16K13742, 
as well as JSPS Joint Research Project with CNRS 
``Zeta functions of several variables and applications,'' 
JSPS Core-to-Core program ``Foundation of a Global Research Cooperative 
Center in Mathematics focused on Number Theory and Geometry'' and 
the KiPAS program 2013--2018 of the Faculty of Science and Technology 
at Keio University.

\end{document}